\documentclass{lmsMODIFIED}

\usepackage{amsmath, amscd, amsfonts, amssymb, graphicx}

\newtheorem{theorem}{Theorem}[section] 
\newtheorem{lemma}[theorem]{Lemma}     

\newtheorem{proposition}[theorem]{Proposition}

\newnumbered{definition}[theorem]{Definition}
\newnumbered{remark}[theorem]{Remark}
\newnumbered{example}[theorem]{Example}

\newcommand{\abs}[1]{\vert#1\vert}

\newcommand{\Comp}{\mathbb{C}}

\newcommand{\ROI}[1]{O_{#1}}
\newcommand{\roi}[1]{\mathcal{O}_{#1}}

\newcommand{\Res}[1]{\overline{#1}}
\newcommand{\res}[1]{\overline{#1}}
\newcommand{\Coset}[2]{#1+t(#2)\roi{F}}
\newcommand{\g}{\gamma}
\newcommand{\calL}{\mathcal{L}}

\newcommand{\Char}[1]{\mbox{char}_{#1}}

\newcommand{\al}{\alpha}
\newcommand{\CG}{\Comp(\Gamma)}

\newcommand{\mult}[1]{#1^{\times}}

\newcommand{\comment}[1]{}

\newcommand{\GL}[2]{\mbox{GL}_{#1}(#2)}
\newcommand{\SL}[2]{\mbox{SL}_{#1}(#2)}

\newcommand{\M}[2]{\mbox{M}_{#1}(#2)}
\newcommand{\no}[1]{\dot{#1}}
\newcommand{\G}{\Gamma}
\newcommand{\ssc}[1]{\mbox{\scriptsize $#1$}}

\title[Integration on product spaces and $\mbox{GL}_n$]{Integration on product spaces and $\mbox{GL}_n$ of a valuation field over a local field}

\author{Matthew T. Morrow}

\classno{11S80 (primary), 20G25, 20G05, 28C10 (secondary)}

\extraline{The author is supported by an EPSRC Doctoral Training Grant at the University of Nottingham.}

\begin{document}

\maketitle

\begin{abstract}
We present elements of a theory of translation-invariant integration on finite dimensional vector spaces and on $\mbox{GL}_n$ over a valuation field with local field as residue field. We then discuss the case of an arbitrary algebraic group. This extends the work of Fesenko.
\end{abstract}

\subsection*{Introduction}
This paper addresses the problem of measure and integration on a finite dimensional vector space and on $\mbox{GL}_n$ over a valuation field whose residue field is a local field. This, and the more fundamental problem of integration over the field itself have been considered by Fesenko \cite{Fesenko-analysis-on-arithmetic-schemes} \cite{Fesenko-vector-spaces} \cite{Fesenko-analysis-on-loop-spaces} and Kim and Lee \cite{Kim-Lee2} \cite{Kim-Lee1} in the case of a higher dimensional local field, and by the author \cite{Morrow_1}. Far more general results of Hrushovski and Kazhdan using model theory \cite{Hrushovski-Kazhdan1} \cite{Hrushovski-Kazhdan2} treat the case in which $\res{F}$ has characteristic zero.

Such a theory has applications in the representation theory of two-dimensional local fields (see \cite{Kim-Lee2}) and related problems in the Langlands programme.

We now outline the contents of the paper.

Let $F$ be a valued field with arbitrary value group and ring of integers $\roi{F}$, whose residue field is a non-discrete locally compact field; let $\CG$ be the field of fractions of the complex group algebra of $\Gamma$. In \cite{Morrow_1} the author used ideas of Fesenko \cite{Fesenko-analysis-on-arithmetic-schemes} \cite{Fesenko-analysis-on-loop-spaces} to introduce elements of a theory of integration over $F$ with values in $\CG$. In the first section, we give a summary of the results required for this paper.

In the second section, the integral on $F$ is extended to $F^n$ using repeated integration. So that Fubini's theorem holds, we consider $\CG$-valued functions $f$ on $F^n$ such that for any permutation $\sigma$ of $\{1,\dots,n\}$ the repeated integral \[\int^F\dots\int^F f(x_1,\dots,x_n)\,dx_{\sigma(1)}\dots \,dx_{\sigma(n)}\] is well defined, and its value does not depend on $\sigma$; such a function is called Fubini.

Now suppose that $g$ is a Schwartz-Bruhat function on $\res{F}^n$; let $f$ be the complex-valued function on $F^n$ which vanishes on $\roi{\res{F}}^n$, and satisfies \[f(x_1,\dots,x_n)=g(\res{x}_1,\dots,\res{x}_n)\] for $x_1,\dots,x_n\in\roi{F}$. $f$ is shown to be Fubini in the second section. In proposition \ref{prop_GL_is_Fubini} it is shown that if $a\in F$ and $\tau\in\GL{N}{F}$, then $x\mapsto f(a+\tau x)$ is also Fubini and \[\int^{F^n} f(a+\tau x)\,dx=\abs{\det\tau}^{-1}\int^{F^n}f(x)\,dx\tag{$\ast$}\] where $\abs{\cdot}$ is an absolute value on $F$. The main result of the third section, theorem \ref{theorem_main_result}, easily follows: there exists a space of Fubini functions $\calL(F^n,\mbox{GL}_n)$ such that $\calL(F^n,\mbox{GL}_n)$ is closed under affine changes of variable, with ($\ast$) holding for $f\in\calL(F^n,\mbox{GL}_n)$.

Next, just as in the classical case of a local field, we look at $\CG$-valued functions $\phi$ on $\GL{n}{F}$, for which $\tau\mapsto\phi(\tau)\abs{\det\tau}^{-n}$ belongs to $\calL(F^{n^2})$, having identified $F^{n^2}$ with the space of $n\times n$ matrices over $F$. This leads to an integral on $\GL{n}{F}$ which is left and right translation invariant, and which lifts the Haar integral on $\GL{n}{\res{F}}$ in a certain sense.

Finally we discuss extending the theory to the case of an arbirary algebraic group.

\subsection*{Acknowledgements}
I. Fesenko provided invaluable help during the writing of this text.

\subsection*{Notation}
Let $\Gamma$ be a totally ordered group and $F$ a field with a valuation $\nu:F^{\times}\to\Gamma$ with residue field $\Res{F}$, ring of integers $\roi{F}$ and residue map $\rho:\roi{F}\to\Res{F}$ (also denoted by an overline). Suppose further that the valuation is split; that is, there exists a homomorphism $t:\Gamma\to F^{\times}$ such that $\nu\circ t=\mbox{id}_\Gamma$.

Sets of the form $\Coset a{\g}$ are called \emph{translated fractional ideals}.

$\CG$ denotes the field of fractions of the complex group algebra $\Gamma$; the basis element of the group algebra corresponding to $\g\in\Gamma$ shall be written as $X^{\g}$ rather than as $\g$. With this notation, $X^{\g}X^{\delta}=X^{\g+\delta}$. Note that if $\Gamma$ is a free abelian group of finite rank $n$, then $\CG$ is isomorphic to the rational function field $\Comp(X_1,\dots,X_n)$.

We fix a choice of Haar measure on $\res{F}$. The measure on $\mult{\res{F}}$ is chosen to satisfy $\mult{d}x=\abs{x}^{-1}d^+x$, and the measure on $\res{F}^m$ is always the product measure.

\begin{remark*}
The assumptions above hold for a higher dimensional local field. For basic definitions and properties of such fields, see \cite{IHLF}.

Indeed, suppose that $F=F_n$ is a higher dimensional local field of dimension $n\ge2$: we allow the case in which $F_1$ is an archimedean local field. If $F_1$ is non-archimedean, instead of the usual rank $n$ valuation $\mathbf{v}:F^{\times}\to\mathbb{Z}^n$, let $\nu$ be the $n-1$ components of $\mathbf{v}$ corresponding to the fields $F_n,\dots,F_2$; note that $\mathbf{v}=(\nu_{\Res{F}}\circ\eta,\nu)$. If $F_1$ is archimedean, then $F$ may be similarly viewed as an valuation field with value group $\mathbb{Z}^{n-1}$ and residue field $F_1$.

The residue field of $F$ with respect to $\nu$ is the local field $\Res{F}=F_1$. If $F$ is non-archimedean, then the ring of integers $\ROI{F}$ of $F$ with respect to the rank $n$ valuation is equal to $\rho^{-1}(\roi{\Res{F}})$, while the groups of units $\mult{\ROI{F}}$ with respect to the rank $n$ valuation is equal to $\rho^{-1}(\roi{\Res{F}}^{\times})$.
\end{remark*}

\section{Integration on $F$}

In \cite{Morrow_1} a theory of integration on $F$ taking values in the field $\CG$ is developed. We repeat here the definitions and main results.

\begin{definition}
For $g$ a function on $\Res{F}$ taking values in an abelian group $A$, set
\begin{align*}
	g^0:F&\to A\\
	x&\mapsto \begin{cases}
		g(\Res{x}) & x\in\roi{F} \\
		0 & \mbox{otherwise.}
		\end{cases}
\end{align*}

More generally, for $a\in F$, $\g\in\Gamma$, the \emph{lift of $g$ at $a,\g$} is the $A$-valued function on $F$ defined by
\[g^{a,\g}(x)=
    \begin{cases}
    g(\Res{(x-a)t(-\g)}) & \quad x\in a+t(\g)\roi{F} \\
    0 & \quad\mbox{otherwise}
    \end{cases}
\]

Note that $g^{0,0}=g^0$ and $g^{a,\g}(a+t(\g)x)=g^0(x)$ for all $x\in F$.
\end{definition}

\begin{definition}
Let $\calL$ denote the space of complex-valued Haar integrable functions on $\res{F}$. A \emph{simple} function on $F$ is a $\CG$-valued function of the form \[x\mapsto g^{a,\g}(x)\,X^{\delta}\] for some $g\in\calL$, $a\in F$, $\g,\delta\in\Gamma$.

Let $\calL(F)$ denote the $\CG$ space of all $\CG$-valued functions spanned by the simple functions; such functions are said to be \emph{integrable} on $F$. 
\end{definition}

\begin{remark}\label{remark_translation_and_scaling_of_integrable_functions}
Note that the space of integrable functions is the smallest $\CG$ space of $\CG$-valued functions on $F$ with the following properties:
\begin{enumerate}
\item If $g\in\calL$, then $g^0\in\calL(F)$.
\item If $f\in\calL(F)$ and $a\in F$ then $\calL(F)$ contains $x\mapsto f(x+a)$.
\item If $f\in\calL(F)$ and $\al\in\mult{F}$ then $\calL(F)$ contains $x\mapsto f(\al x)$.
\end{enumerate}

In fact, it is clear that if $f$ is simple then for $a\in F$ and $\al\in\mult{F}$, the functions $x\mapsto f(x+a)$ and $x\mapsto f(\al x)$ are also simple.
\end{remark}

The main result on existence and properties of an integral is as follows:

\begin{theorem}\label{theorem_main_properties_of_integral}
There is a unique $\CG$-linear functional $\int^F$ on $\calL(F)$ which satisfies 
\begin{enumerate}
\item $\int^F$ lifts the usual integral on $\res{F}$: for $g\in\calL$, \[\int^F(g^0)=\int g(u)\,du;\]
\item Translation invariance: for $f\in\calL(F)$, $a\in F$, \[\int^F f(x+a)\,dx=\int^F f(x)\,dx;\]
\item Compatibility with multiplicative structure: for $f\in\calL(F)$, $\al\in\mult{F}$, \[\int^F f(\al x)\,dx=\abs{\al}^{-1}\int^F f(x)\,dx.\]
\end{enumerate}
Here the \emph{absolute value} of $\al$ is defined by $\abs{\al}=\abs{\al t(-\nu(\al))}X^{\nu(\al)}$, and we have adopted the customary integral notation $\int^F(f)=\int^F f(x)\,dx$.
\end{theorem}
\begin{proof}
See \cite{Morrow_1}.
\end{proof}

\begin{remark}\label{remark_repeated_integral_of_lifted_function}
If $g^{a,n}$ is the lift of a Haar integrable function, then  \[\int^Fg^{a,n}(x)\,dx=\int g(u)\,du\,X^n\].
\end{remark}
%
%

\section{Repeated Integration on $F^n$}

In this section we extend the integral on $F$ to the product space $F^n$ for $n$ a positive integer. We do this by using the integral over $F$ to define repeated integrals. The idea is simple, though the notation is not.

Given a sequence $x_1,\dots,x_n$ of $n$ terms, and $r$ such that $1\le r\le n$, the notation \[x_1,\dots,\no{x}_r,\dots,x_n\quad=\quad x_1,\dots,x_{r-1},x_{r+1},\dots,x_n\] denotes the sequence of $n-1$ terms obtained by removing the $r^{\mbox{\tiny th}}$ term.

We introduce the largest space of functions for which all repeated integrals exist and are equal:

\begin{definition}\label{definition_Fubini}
Let $f$ be a $\CG$-valued function on $F^n$. The inductive definition of $f$ being \emph{Fubini}, and the \emph{repeated integral} of $f$, are as follows:

If $n=1$, then $f$ is Fubini if and only if it is integrable, and the repeated integral of $f$ is defined to be its integral $\int^F f(x)\,dx$.

For $n>1$, $f$ is Fubini if and only if it satisfies the following conditions:
\begin{enumerate}
\item For each $r$ with $1\le r\le n$, and all $x_1,\dots,\no{x}_r,\dots,x_n$ in $F$, the function \[x_r\mapsto f(x_1,\dots,x_n)\] is required to be integrable on $F$, and then the function \[(x_1,\dots,\no{x}_r,\dots,x_n)\mapsto \int f(x_1,\dots,x_n)\,dx_r\] is is required to be Fubini on $F^{n-1}$.
\item Then we require that the repeated integral of $(x_1,\dots,\no{x}_r,\dots,x_n)\mapsto \int f(x_1,\dots,x_n)\,dx_r$ does not depend on $r$. The repeated integral of $f$ on $F^n$ is defined to be the common value of these $n$ repeated integrals on $F^{n-1}$.
\end{enumerate}

The repeated integral of a Fubini function $f$ on $F^n$  will be denoted $\int^{F^n} f(x)\,dx$.
\end{definition}

The repeated integral is a $\CG$-linear functional on the $\CG$-space of all Fubini functions on $F^n$.

\begin{remark}
Informally, a $\CG$-valued function $f$ is Fubini if and only if, for each permutation $\sigma$ of $\{1,\dots,n\}$, the expression \[\int^F\dots\int^F f(x_1,\dots,x_n)\,dx_{\sigma(1)}\dots \,dx_{\sigma(n)}\] is well defined and its value does not depend on $\sigma$. The repeated integral of $f$ is of course the common value of these $n!$ integrals.
\end{remark}

\begin{remark}\label{remarks_on_Fubini_functions}
We will also be interested in repeated integrals of complex-valued functions on $\res{F}^n$. Since the integration theory on $F$ does not allow for functions on $\res{F}$ which are perhaps only defined off a null set, we must ensure that such functions do not arise. Therefore we define a complex-valued function on $\res{F}^n$ to be Fubini if it is Haar integrable and satisfies the obvious rewording of definition \ref{definition_Fubini}. Informally, such a function if Fubini if and only if it is Haar integrable and each partial integral \[\int\dots\int f(x_1,\dots,x_n)\,dx_{\sigma(1)}\dots dx_{\sigma(r)}\] is defined for \emph{all} $x_{\sigma(r+1)}\dots x_{\sigma(n)}\in \res{F}$, where $\sigma$ is any permutation of $\{1,\dots,n\}$ and $1\le r\le n$. Fubini's theorem then implies that the value of the repeated integral \[\int^F\dots\int^F f(x_1 ,\dots,x_n)\,dx_{\sigma(1)}\dots \,dx_{\sigma(n)}\] is independent of $\sigma$.

Fubini's theorem and induction on $n$ imply that any integrable function on $\res{F}^n$ is almost everywhere equal to a Fubini function.

Any continuous complex-valued function on $\res{F}$ with compact support is Fubini, as is any Schwartz function if $\res{F}$ is archimedean. So the class of Fubini functions is still large enough for applications in representation theory, harmonic analysis, etc.
\end{remark}

In fact, most Fubini functions on $F^n$ encountered in this paper will be of the following form, which is generalisation of the notion of a simple function on $F$:

\begin{definition}
Let $f$ be a Fubini function on $F^n$; the inductive definition of $f$ being \emph{strongly Fubini} is as follows:

If $n=1$, then $g$ is strongly Fubini if and only if a simple function.

For $n>1$, $g$ is strongly Fubini if and only if the following holds: For each $r$ with $1\le r\le n$, and each $x_1,\dots,\no{x}_r,\dots,x_n$ in $F$, we require that \[x_r\mapsto f(x_1,\dots,x_n)\] is a simple function on $\res{F}$, and then that \[(x_1,\dots,\no{x}_r,\dots,x_n)\mapsto \int^F f(x_1,\dots,x_n)\,dx_r\] is strongly Fubini on $F^{n-1}$.
\end{definition}

The property of being strongly Fubini is preserved under translation and scaling, as is the weaker property of being Fubini. For $\al=(\al_1,\dots,\al_n)$ in ${\mult{F}}^n$ ($n$ copies of $\mult{F}$, not the group of $n^{\mbox{\tiny th}}$ powers of $\mult{F}$), write $\abs{\al}=\prod_i\abs{\al_i}$, where $\abs{\cdot}$ is the absolute value introduced in theorem \ref{theorem_main_properties_of_integral}.

\begin{lemma}\label{lemma_translation_and_scaling_of_Fubini_functions}
Suppose $f$ is a strongly Fubini (resp. Fubini) function on $F^n$. For $a\in F^n$ and $\al\in{\mult{F}}^n$, the functions $x\mapsto f(x+a)$ and $x\mapsto f(\al x)$ are strongly Fubini (resp. Fubini), with repeated integrals
\[\int^{F^n}f(x+a)\,dx=\int^{F^n}f(x)\,dx,\quad\quad\int^{F^n}f(\al x)\,dx=\abs{\al}^{-1}\int^{F^n}f(x)\,dx.\]
\end{lemma}
\begin{proof}
This is a simple induction on $n$; the case $n=1$ is remark \ref{remark_translation_and_scaling_of_integrable_functions}.
\end{proof}

A theme of this paper is showing how integrals constructed at the level of $F$ lift Haar integrals on $\res{F}$. For the integral on $F$, this is the identity \[\int^F g^0(x)\,dx=\int g(u)\,du\] for Haar integrable $g$ on $\res{F}$.

We will denote by $t:\G^n\to F^n$ the product of $n$ copies of $t$; the value of $n$ will be clear from the context. Similarly, we write $\rho$ or an overline for the the residue map $\roi{F}^n\to\res{F}^n$. Given $a=(a_1,\dots,a_n)\in F^n$ and $\g=(\g_1,\dots,\g_n)\in\G$, there is a product of translated fractional ideals given by \[a+t(\g)\roi{F}^n=\prod_{i=1}^n a_i+t(\g_i)\roi{F}.\]

Now we may generalise the notion of lifting a function:

\begin{definition}
For $g$ a function on $\Res{F}^n$ taking values in an abelian group $A$, set
\begin{align*}
	g^0:F^n&\to A\\
	x&\mapsto \begin{cases}
		g(\Res{x}) & x\in\roi{F}^n \\
		0 & \mbox{otherwise.}
		\end{cases}
\end{align*}

Again, more generally, for $a\in F^n$, $\g\in\Gamma^n$, the \emph{lift of $g$ at $a,\g$} is the $A$-valued function on $F$ defined by
\[g^{a,\g}(x)=
    \begin{cases}
    g(\Res{(x-a)t(-\g)}) & \quad x\in a+t(\g)\roi{F}^n\\
    0 & \quad\mbox{otherwise}
    \end{cases}
\]

Of course, $g^0=g^{0,0}$ and $g^{a,\g}(a+t(\g)x)=g^{0,0}(x)$ for all $x\in F^n$
\end{definition}

\begin{remark}\label{remark_section_of_lifted_function}
It is a straightforward observation that a section of a lifted function is again a lifted function. To be precise, suppose that $f=g^{a,\g}$ is a lifted function as in the definition, $r$ is such that $1\le r\le n$, and $x_1,\dots,\no{x}_r,\dots,x_n\in F$. Then the function \[x_r\mapsto f(x_1,\dots,x_n)\] of $F$ is identically zero unless $x_i\in a_i+t(\g_i)\roi{F}$ for all $i\neq r$.

If in fact $x_i\in a_i+t(\g_i)\roi{F}$ for all $i\neq r$, then \[x_r\mapsto f(x_1,\dots,x_n)\] is the lift of \[u_r\mapsto g(\res{\xi}_1,\dots,\res{\xi}_{r-1}, u_r,\res{\xi}_{r+1},\dots,\res{\xi}_n)\] at $a_r,\g_r$, where $\xi_i:=(x_i-a_i)t(-\g_i)\in\roi{F}$ for $i\neq r$.

This generalises to $s$-dimensional sections of $f$ for any $s$ with $1\le s\le n$. We shall frequently employ the cases $s=1$ and $s=2$.
\end{remark}

We may now prove the fundamental result that the repeated integral on $F^n$ lifts the Haar integral on $\res{F}^n$:

\begin{proposition}\label{proposition_integral_of_lifted_function}
Suppose $g$ is a Fubini function on $F^n$. Then $g^0$ is strongly Fubini on $F^n$, with repeated integral \[\int^{F^n}f(x)\,dx=\int_{\res{F}^n} g(u)\,du\,X^{\sum_{i=1}^n \g_i}.\]
\end{proposition}
\begin{proof}
Let $r$ be such that $1\le r\le n$, and fix $x_1,\dots,\no{x}_r,\dots,x_n\in F$. The previous remark and the case $n=1$ (contained in theorem \ref{theorem_main_properties_of_integral}) imply that $x_r\mapsto g^0(x_1,\dots,x_n)$ is simple and integrable on $F$ with integral
\[\begin{cases}
	\int^F g(\res{x}_1,\dots,\res{x}_{r-1},u_r,\res{x}_{r+1},\dots,\res{x}_n)\,dx_r & x_i\in\roi{F}\mbox{ for all }i\neq r\\
	0 & \mbox{otherwise}.
	\end{cases}\]
That is, \[(x_1,\dots,\no{x}_r,\dots,x_n)\mapsto\int^F g^0(x_1,\dots,x_n)\,dx_r\] is the lift of the everywhere defined Haar integrable function \[(u_1,\dots,u_n)\mapsto\int g(u_1,\dots,u_n)\,du_r\] on $F^{n-1}$.

The result now follows easily by induction on $n$.
\end{proof}

\begin{remark}\label{remark_integral_of_lifted_function}
More generally, suppose $f=g^{a,\g}$ is the lift of a Fubini function to $F^n$; here $g$ is Fubini on $\res{F}^n$, $a\in F$ and $\g\in\Gamma$. Then the proposition, and the invariance of being strongly Fubini under translation and scaling (lemma \ref{lemma_translation_and_scaling_of_Fubini_functions}) imply $f$ is strongly Fubini on $F^n$, with repeated integral \[\int^{F^n}f(x)\,dx=\int_{\res{F}^n} g(u)\,du\,X^{\sum_{i=1}^n \g_i}.\]
\end{remark}
%
%

\section{Change of variables from $\mbox{GL}_n$ in repeated integrals}
With the basics of repeated integrals in place, we turn to the interaction of the theory with $\GL{n}{F}$. We shall write the action of $\GL{n}{F}$ on $F^n$ as a left action, though we also write elements of $F^n$ as row vectors; given $\tau\in\GL{n}{F}$ and $x=(x_1,\dots,x_n)\in F^n$, $\tau x$ means
\[\tau x = \tau \left(\begin{array}{c}x_1\\ \vdots\\ x_n\end{array}\right).\]
Given a function $f$ on $F^n$, we write $f\circ\tau$ for the function $x\mapsto f(\tau x)$. $\SL{n}{F}$ denotes the determinant $1$ subgroup of $\GL{n}{F}$. These notation also apply to $\res{F}$ in place of $F$.

\begin{definition}
A complex-valued function $f$ on $\res{F}^n$ is said to be \emph{GL-Fubini} if and only if $f\circ\tau$ is Fubini for all $\tau\in\GL{n}{\res{F}}$.
\end{definition}

\begin{remark} \label{remark_on_GL_Fubini_functions}
Any continuous complex-valued function with compact support is GL-Fubini, as is any Schwartz function when $\res{F}$ is archimedean; this follows from remark \ref{remarks_on_Fubini_functions} and the invariance of these properties under $\GL{n}{\res{F}}$. In the following results this is the sort of function to have in mind.
\end{remark}

\begin{definition}
Let $\calL(F^n, \mbox{GL}_n)$ be the $\CG$ space of $\CG$-valued functions spanned by $g^{a,\g}\circ\tau$ for $g$ GL-Fubini, $\tau\in\GL{n}{F}$, $a\in F^n$, $\g\in\Gamma^n$.
\end{definition}

The aim of this section is the following result:

\begin{theorem}\label{theorem_main_result}
Every function in $\calL(F^n, \mbox{GL}_n)$ is Fubini on $F^n$. If $f\in\calL(F^n, \mbox{GL}_n)$, $a\in F^n$, and $\tau\in\GL{n}{F}$, then the functions $x\mapsto f(x+a)$ and $x\mapsto f(\tau x)$ belong to $\calL(F^n, \mbox{GL}_n)$, with repeated integrals given by
\[\int^{F^n}f(x+a)\,dx=\int^{F^n}f(x)\,dx,\quad\quad
    \int^{F^n}f(\tau x)\,dx=\abs{\det\tau}^{-1}\int^{F^n}f(x)\,dx\]
\end{theorem}

The theorem will be proved through several smaller results. First we recall the Iwasawa decomposition, where we abbreviate "unipotent upper triangular" to u.u.t.

\begin{lemma} \label{lemma_Iwasawa_decomposition}
Let $\tau$ be in $\GL{n}{F}$. Then there exist $A$ in $\GL{n}{\roi{F}}$, a u.u.t. $U$ in $\GL{n}{F}$, and a diagonal $\Lambda$ in $\GL{n}{F}$ such that $\tau=AU\Lambda$.
\end{lemma}
\begin{proof}
When $\Gamma=\mathbb{Z}$ and $F$ is complete with respect to the discrete valuation $\nu$, this is the standard Iwasawa decomposition. However, the standard proof is valid in the generality in which we require it (see eg. \cite{Bump}).
\end{proof}

This decomposition allows us to restrict attention to a upper triangular matrices, for the $\GL{n}{\roi{F}}$ term can be "absorbed" into the function:

\begin{lemma}\label{lemma_corollary_to_Iwasawa_decomposition}
$\calL(F^n, \mbox{GL}_n)$ is spanned over $\CG$ by functions of the form $x\mapsto g^0\circ U(\al x+a)$, for $g$ GL-Fubini on $\res{F}^n$, $U$ a u.u.t. matrix, $\al\in{\mult{F}}^n$, and $a\in F^n$.
\end{lemma}
\begin{proof}
Let $g$ be GL-Fubini on $\res{F}^n$; let $a\in F^n$, $\g\in\Gamma^n$. Let $A,U,\Lambda$ be the decomposition of \[\left(\begin{array}{ccc}t(-\g_1)&&\\&\ddots&\\&&t(-\g_n)\end{array}\right)U\] as in lemma \ref{lemma_Iwasawa_decomposition}. For $x$ in $F$, the identity $g^{a,\g}(x)=g^0\circ AU\Lambda(x-\tau^{-1}a)$ holds.

Now note that $g^0\circ A=(g\circ\res{A})^0$ where $\res{A}$ is the image of $A$ in $\GL{n}{\res{F}}$. So $x\in F$ implies $g^{a,\g}(x)=(g\circ \res{A})^0( U(\lambda x+b))$, where $\lambda\in{\mult{F}}^n$ is defined by \[\Lambda=\left(\begin{array}{ccc}\lambda_1&&\\&\ddots&\\&&\lambda_n\end{array}\right),\] and $b=-\lambda\tau^{-1} a$.
\end{proof}

We now turn to proving special cases of the main theorem as well as some technical lemmas. Particular attention is given to the case $n=2$, for it is required several times later in inductions.

\begin{lemma} \label{Fubini_statement}
Let $g$ be GL-Fubini on $\res{F}^2$ and set $f=g^0$. Let $\al\in F$ and set $e=\al^{-1}t(\nu(\al))$ if $\al\neq0$ and $e=0$ otherwise; set $\delta_0=\min(\nu(\al),0)$.

There exists $\tau\in\SL{2}{\res{F}}$, independent of $g$, such that, for any $x\in F$, the function $y\mapsto f(x+\al y,y)$ equals
\[\begin{cases}
    \mbox{the lift of $v\mapsto g\circ\tau(\res{xt(-\delta_0)},v)$ at
    $-xet(-\delta_0),-\delta_0$} &\mbox{if } x\in t(\delta_0)\roi{F}\\
    0 & \mbox{otherwise.}
\end{cases}\]
\end{lemma}
\begin{proof}
If $\al=0$ then we are just considering a section of a Fubini function and so $\tau=\mbox{id}$ suffices by remark \ref{remark_section_of_lifted_function}. Henceforth assume that $\al\neq0$.

We first consider the case $\al=t(\delta)$ for some $\delta\in\G$; so $e=1$. Consider, for any $x\in F$, the section
\begin{align*}
	D_x:F&\to\Comp\\
	y&\mapsto f(x+t(\delta)y,y).
	\end{align*}
We make the following claim, dependent on the sign of $\delta$, regarding $D_x$:

{\bf case: $\delta<0$.}
\[D_x=\begin{cases}
    \mbox{lift of $v\mapsto g(v,-\res{xt(-\delta}))$ at
    $-xt(-\delta),-\delta$} &\mbox{if } x\in t(\delta)\roi{F} \\
    0 & \mbox{otherwise.}
\end{cases}\]

{\bf case: $\delta=0$.}
\[D_x=\begin{cases}
    \mbox{lift of $v\mapsto g(v+\res{x},v)$ at
    $0,0$} &\mbox{if } x\in\roi{F} \\
    0 & \mbox{otherwise}
\end{cases}\]

{\bf case: $\delta>0$.}
\[D_x=\begin{cases}
    \mbox{lift of $v\mapsto g(\res{x},v)$ at
    $0,0$} &\mbox{if } x\in\roi{F} \\
    0 & \mbox{otherwise.}
\end{cases}\]

We shall prove the case $\delta=0$. For any $x,y\in F$, $f(x+y,y)$ vanishes unless $x+y$ and $y$ both belong to $\roi{F}$; hence $D_x$ is identically zero unless $x\in\roi{F}$. Assuming that $x\in\roi{F}$, it remains to verify that \[D_x=\mbox{lift of $v\mapsto g(v+\res{x},v)$ at $0,0$}.\] Both sides vanish off $\roi{F}$ and are seen to agree on $\roi{F}$ by direct evaluation. This proves the claim in this case. The other cases are proved by similar arguments and we omit the details.

If $\delta\ge 0$ and $x\in\roi{F}$, then $D_x$ is also the lift of a function at $-x,0$:

{\bf case: $\delta=0$.}
\[D_x=\mbox{lift of $v\mapsto g(v,v-\res{x})$ at $-x,0$}\]

{\bf case: $\delta>0$.}
\[D_x=\mbox{lift of $v\mapsto g(\res{x},v-\res{x})$ at $-x,0$}\]
The proof when $\al\in t(\Gamma)$ is completed by setting:

{\bf case: $\delta<0$.} \[\tau=\left(\begin{array}{cc}0&1\\-1&0\end{array}\right)\]

{\bf case: $\delta=0$.} \[\tau=\left(\begin{array}{cc}0&1\\-1&1\end{array}\right)\]

{\bf case: $\delta>0$.} \[\tau=\left(\begin{array}{cc}1&0\\-1&1\end{array}\right)\]

In the general case, write $\al=e^{-1}t(\delta)$, with $\delta=\nu(\al)$ and $e\in\mult{\roi{F}}$; let $\tau'=\left(\begin{array}{cc} \res{e}^{-1}&0\\0&1 \end{array}\right)$. Also introduce $f'(x,y)=f(e^{-1}x,y)$, which is the lift of $(u,v)\mapsto g(\res{e}^{-1}u,v)=g\circ\tau'(u,v)$ (a Fubini function on $\res{F}^2$) at $0,0$. By the case above, there exists $\tau\in\SL{2}{\res{F}}$ such that $x\in F$ implies $y\mapsto f'(x+t(\delta) y,y)=f(e^{-1}x+\al y,y)$ equals
\[\begin{cases}
    \mbox{the lift of $v\mapsto g\circ\tau'\tau(\res{xt(-\delta_0)},v)$ at
    $-xt(-\delta_0),-\delta_0$} &\mbox{if } \nu(x)\ge\delta_0\\
    0 & \mbox{otherwise.}
\end{cases}\]
Hence $y\mapsto f(x+\al y,y)=f'(ex+\al y,y)$ equals
\[\begin{cases}
    \mbox{the lift of $v\mapsto
    g\circ\tau'\tau(\res{e}\;\res{xt(-\delta_0)},v)$ at
    $-ext(-\delta_0),-\delta_0$} &\mbox{if } \nu(x)\ge\delta_0\\
    0 & \mbox{otherwise.}
\end{cases}\]
As $\tau'\tau\left(\begin{array}{cc} \res{e}&0\\0&1 \end{array}\right)$ has determinant $1$, this completes the proof.
\end{proof}

Remaining with the case $n=2$, we now extend the previous lemma slightly in preparation for the induction on $n$:

\begin{lemma}
Let $g$ be GL-Fubini on $\res{F}^2$, $a\in F$, $\g\in\Gamma$; set $f=g^{(0,a),(0,\g)}$. Let $\al\in F$ and set $\delta=\min(\nu(\al)+\g,0)$.

There exist $b,c\in F$ (independent of $g$) and $\tau\in\SL{2}{\res{F}}$ (independent of $g$ and $a$) such that  $x\in F$ implies $y\mapsto f(x+\al y, y)$ equals
\[\begin{cases}\mbox{the lift of
    $v\mapsto g\circ\tau(\res{(x-c)t(-\delta)},v)$ at $b,\g-\delta$}
    &\mbox{if $x\in c+t(\delta)\roi{F}$}\\
    0&\mbox{otherwise}
\end{cases}\]
\end{lemma}
\begin{proof}
Let $e=\al^{-1}t(\nu(\al))$ if $\al\neq 0$ and $e=0$ otherwise. For $x$ in $F$ the previous lemma implies that $y\mapsto g^0(x+t(\g)\al y,y)$ equals
\[\begin{cases}
    \mbox{the lift of $v\mapsto g\circ\tau(\res{xt(-\delta)},v)$ at
    $-xet(-\delta),-\delta$} &\mbox{if } x\in t(\delta)\roi{F}\\
    0 & \mbox{otherwise}
\end{cases}\]
for some $\tau\in\SL{2}{\res{F}}$ (independent of $g$ by the previous lemma, and clearly independent of $a$).

For $x,y\in F$, the identity
\begin{align*}f&(x+\al y,y)\\
    &=g^0(x+\al y, (y-a)t(-\g))\\
    &=g^0(x+\al a+t(\g)\al(y-a)t(-\g),(y-a)t(-\g))\\
    &=\begin{cases}
    g\circ\tau(\res{(x+\al a)t(-\delta)},
        \res{((y-a)t(-\g)+xet(-\delta))t(\delta)})
        &\mbox{if } x+\al a\in t(\delta)\roi{F}\\
    0&\mbox{otherwise}
    \end{cases}
\end{align*}
follows. Set $b=a-ext(-\delta)$ and $c=-\al a$ to complete the proof.
\end{proof}

The following result extends the previous lemma to the case of arbitrary $n\ge2$; it is rather technical.

\begin{lemma}
Let $g$ be GL-Fubini on $\res{F}^n$, $a\in F$, $\g\in\G$; set $f=g^{(0,\dots,0,a),(0,\dots,0,\g)}$. Let $\al_i\in F$ for $1\le i\le n-1$. Then
\begin{enumerate}
\item For all $x_1,\dots,x_{n-1}\in F$, the function of $F$ \[x_n\mapsto f(x_1+\al_1 x_n,\dots,x_{n-1}+\al_{n-1} x_n, x_n)\] is integrable and simple.

\item Further, there exist $\tau\in\SL{n}{\res{F}}$, $\delta\in\G^{n-1}$, and $c\in F^{n-1}$ such that the function of $F^{n-1}$ \[(x_1,\dots,x_{n-1})\mapsto\int f(x_1+\al_1 x_n,\dots,x_{n-1}+\al_{n-1} x_n, x_n)\,dx_n\] is the lift of \[(u_1,\dots,u_{n-1})\mapsto \int g\circ\tau(u_1,\dots,u_n)\,du_n\;X^{\g-\sum_{i=1}^{n-1}\delta_i}\] at $b,\delta$. Also, $\tau$ may be chosen to be independent of $g$ and $a$.
\end{enumerate}
\end{lemma}
\begin{proof}
The proof is by induction on $n$.

Let $\delta_{n-1}=\min(\nu(\al_{n-1})+\g,0)$. Let $\xi_1,\dots,\xi_{n-2}$ be in $\roi{F}$; the function \[(x_{n-1},x_n)\mapsto f(\xi_1,\dots,\xi_{n-2},x_{n-1},x_n)\] is the lift of \[(u_{n-1},u_n)\mapsto g(\res{\xi}_1,\dots,\res{\xi}_{n-2},u_{n-1},u_n),\] which is GL-Fubini, at $(0,a),(0,\g)$; this is just a generalisation of remark \ref{remark_section_of_lifted_function} to a two dimensional section. By the previous lemma, there exist $b,c_{n-1}\in F$ and $\tau\in\SL{2}{\res{F}}$, all independent of $\xi_1,\dots,\xi_{n-2}$, such that for all $x_{n-1}\in F$, \[x_n\mapsto f(\xi_1,\dots,\xi_{n-2},x_{n-1}+\al_{n-2}x_n,x_n)\] equals the lift of \[u_n\mapsto
    g(\res{\xi}_1,\dots,\res{\xi}_{n-2},\tau(\res{(x_{n-1}-c_{n-1})t(-\delta_{n-1})}),u_n)\] at $b,\g-\delta_{n-1}$ if $x_{n-1}\in c_{n-1}+t(\delta_{n-1})\roi{F}$, and equals $0$ otherwise.

Also denote by $\tau$ the element of $\SL{n}{\res{F}}$ given by $\left(\begin{array}{cc} I_{n-2}&0\\0&\tau\end{array}\right)$, where $I_{n-2}$ denotes the $n-2$ by $n-2$ identity matrix.

Now take $\xi_{n-1}\in c_{n-1}+t(\delta_{n-1})\roi{F}$; so $\xi_{n-1}=c_{n-1}+t(\delta_{n-1})\xi_{n-1}'$, say. It has been shown that \[(x_1,\dots,x_{n-2},x_n)\mapsto f(x_1,\dots,x_{n-2},\xi_{n-1}+\al_{n-1},x_n)\] is the lift of \[(u_1,\dots,u_{n-2},u_n)\mapsto g\circ\tau(u_1,\dots,u_{n-2},\res{\xi}_{n-1}',u_n),\] which is GL-Fubini, at $(0,\dots,0,b), (0,\dots,0,\g-\delta_{n-1})$. By the inductive hypothesis, the following hold:
\begin{enumerate}
\item For all $x_1,\dots,x_{n-2}\in F$, \[x_n\mapsto f(x_1+\al_1 x_n,\dots,\xi_{n-1}+\al_{n-1},x_n)\] is the lift of a simple integrable function.
\item There exists $\tau'\in\SL{n-1}{\res{F}}$ (independent of $\xi_{n-1}$, $g$, $b$) and $\delta_i\in\G$, $c_i\in F$ ($1\le i\le n-2$), such that \[(x_1,\dots,x_{n-2})\mapsto\int f(x_1+\al_1x_n,\dots,\xi_{n-1}+\al_{n-1},x_n)\,dx_n\] is the lift of \[(u_1,\dots,u_{n-2})\mapsto\int g\circ\tau\tau'(u_1,\dots,u_{n-2},\res{\xi_{n-1}},u_n)\,du_n\;X^{\g-\delta_{n-1}-\sum_{i=1}^{n-2}\delta_i}\] at $(c_1,\dots,c_{n-2}),(\delta_1,\dots,\delta_{n-2})$.
\end{enumerate}

It follows that
\begin{enumerate}
\item
For any $x_1,\dots,x_{n-1}$ in $F$, \[x_n\mapsto f(x_1+\al_1 x_n,\dots,x_{n-1}+\al_{n-1},x_n)\] is a simple integrable function (this function is zero unless $x_{n-1}\in-\al a+t(\delta_{n-1})\roi{F}$, in which case the statement follows from (i) above).
\item The function
\[(x_1,\dots,x_{n-1})\mapsto\int f(x_1+\al_1x_n,\dots,x_{n-1}+\al_{n-1}x_n,x_n)\,dx_n\]
is the lift of
\[(u_1,\dots,u_{n-1})\mapsto\int g\circ\tau\tau'(u_1,\dots,u_n)\,du_n\;X^{\g-\sum_{i=1}^{n-1}\delta_i}\] at $(c_1,\dots,c_{n-1}),(\delta_1,\dots,\delta_{n-1})$.
\end{enumerate}
This completes the proof.
\end{proof}

The following lemma was concerned with the case of a matrix differing from the identity only along the left-most column. We now consider the case of an arbitrary u.u.t. matrix:

\begin{proposition}\label{prop_unipotent_is_Fubini}
Suppose $g$ is GL-Fubini on $\res{F}^n$, $a\in F^n$, $\g\in\Gamma^n$, $\delta\in\Gamma$; set $f=g^{a,\g}\,X^{\delta}$. Let $U$ be a u.u.t. matrix in $\GL{n}{F}$. Then $f\circ U$ is strongly Fubini on $F^n$, with
\[\int^{F^n} f\circ U(x)\,dx=\int^{F^n} f (x)\,dx.\]
\end{proposition}
\begin{proof}
The proof is by induction on $n$.

For any $n$, we claim that it suffices to prove the special case $a=0$, $\g=0$, $\delta=0$. We may clearly assume $\delta=0$ by linearity. For $x\in F^n$ the identity
\begin{align*}
f(Ux)=g^0(Ux)&=g^{0,0}((Ux-a)t(-\g))\\
    &=g^0\circ U_1(t(-\g)(x-U^{-1}a))
\end{align*}
holds, where $U_1$ is the u.u.t. matrix
\[U_1=
    \left(\begin{array}{ccc}t(-\g_1)&&\\&\ddots&\\&&t(-\g_n)\end{array}\right)
    U
    \left(\begin{array}{ccc}t(\g_1)&&\\&\ddots&\\&&t(\g_n)\end{array}\right)
\]
The special case implies that $g^0\circ U_1$ is strongly Fubini with repeated integral equal to that of $g^0$. Thus $f\circ U$ differs from a strongly Fubini function by translation and scaling and hence is itself strongly Fubini (lemma \ref{lemma_translation_and_scaling_of_Fubini_functions}) is therefore strongly Fubini while compatibility with the repeated integral on $F^n$ and Haar integral on $\res{F}^n$ (proposition \ref{proposition_integral_of_lifted_function}) implies
\begin{align*}
    \int^{F^n} f\circ U(x)\,dx&=\abs{t(\g)} \int^{F^n} g^0(x) \,dx\\
    &=X^{\sum_{i=1}^n \g_i}\int_{\res{F}^n}g(u)\,du\\
    &=\int^{F^n} f(x) \,dx
\end{align*}
This completes the proof of the claim.

For each $r$ with $1\le r\le n$, we must now prove that
\begin{enumerate}
\item For $x_1,\dots,\no{x}_r,\dots,x_n\in F$, the function of $F$ $x_r\mapsto f\circ U(x_1,\dots,x_n)$ is simple and integrable.
\item The function of $F^{n-1}$ \[(x_1,\dots,\no{x}_r,\dots,x_n)\mapsto \int f\circ U(x_1,\dots,x_n)\,dx_r\] is strongly Fubini, with repeated integral equal to that of $f$.
\end{enumerate}

The inductive step depends on decomposing $U$ in a certain way. Write
\[U=\left(\begin{array}{ccccc}
    1&\al_{1,2}&\cdots&\al_{1,n}\\
     & \ddots& \ddots& \vdots\\
     & & \ddots& \al_{n-1,n}\\
     & & & 1
\end{array}\right)\]
and observe that $U(x_1,\dots,x_n) =(x_1+\sum_{i=2}^n\al_{1,i},\dots,x_{n-1}+\al_{n-1,n}x_n,x_n)$. Let $V$ be the u.u.t. matrix obtained by setting to zero all entries in the $r^{\mbox{\tiny th}}$ row and $r^{\mbox{\tiny th}}$ column of $U$, apart from the $1$ in the $r,r$-place. Let $V'$ be the $n-1$ by $n-1$ u.u.t. matrix obtained by removing the $r^{\mbox{\tiny th}}$ row and $r^{\mbox{\tiny th}}$ column of $U$. There exist $\beta_{r+1},\dots,\beta_n\in F$ such that the u.u.t. matrix $P$ defined by
\[P(x_1,\dots,x_n)=(x_1+\al_{1,r}x_r,\dots,x_{r-1}+\al_{r-1,r}x_r,x_r+\sum_{i=r+1}^n\beta_i x_i,x_{r+1},\dots,x_n)\]
satisfies $U=PV$.

We are now equipped to begin the main part of the proof. The previous lemma (if $r>1$; it follows straight from the definition of a strongly Fubini function if $r=1$) implies that for fixed $x_1,\dots,\no{x}_r,\dots,x_n\in F$, the function
\[x_r\mapsto f((x_1-\al_{1,r}\sum_{i=r+1}^n\beta_i x_i)+\al_{1,r}x_r, \dots,(x_{r-1}-\al_{r-1,r}\sum_{i=r+1}^n\beta_i x_i)+\al_{r-1,r}x_r, x_r,\dots,x_n)\]
is simple and integrable on $F$. Therefore
\begin{align*}
	x_r\mapsto &f(x_1+\al_{1,r}x_r,\dots,x_{r-1}+\al_{r-1,r}x_r,x_r+\sum_{i=r+1}^n\beta_i x_i,x_{r+1}\dots,x_n)\\
	=&f\circ P(x_1,\dots,x_n)
	\end{align*}
is a translate of a simple integrable function and hence is simple and integrable by remark \ref{remark_translation_and_scaling_of_integrable_functions}. Replacing $x_1,\dots,\no{x}_r,\dots,x_n$ by $V'(x_1,\dots,\no{x}_r,\dots,x_n)$ implies that the function
\begin{align*}
	x_r\mapsto &f\circ PV(x_1,\dots,x_n)\\ &=f\circ U(x_1,\dots,x_n)
\end{align*}
is simple and integrable, proving (i).

The previous lemma (if $r>1$) and translation invariance (any $r$) of the integral also imply that \[f':(x_1,\dots,\no{x}_r,\dots,x_n)\mapsto \int f\circ P(x_1,\dots,x_n)\,dx_r\] is the lift of \[(u_1,\dots,\no{u}_r,\dots,u_n)\mapsto\int g\circ\tau(u_1,\dots,u_n)\,du_r \;X^{-\sum_{i=1}^{n-1}\delta_i}\] at $b,\delta$ for some $b\in F^{n-1}$, $\delta=(\delta_i)\in\G^{n-1}$.

The inductive hypothesis with function $f'$ and matrix $V'$ implies that $f'\circ V'$ is strongly Fubini with repeated integral equal to that of $f'$. But the repeated integral of $f'$ is
\begin{align*}
	\int_{\res{F}^n} g\circ\tau(u)\,du\;X^{-\sum_{i=1}^{n-1}\delta_i}\;X^{\sum_{i=1}^{n-1}\delta_i}
	&=\int_{\res{F}^n} g(u) \,du\\
	&=\int^{F^n} f(x)\,dx
	\end{align*}
by remark \ref{remark_integral_of_lifted_function}, and
\begin{align*}
	f'\circ V'(x_1,\dots,\no{x}_r,\dots,x_n)
	&=\int f\circ P V (x_1,\dots,x_n)\,dx_r\\
	&=\int f\circ U(x_1,\dots,x_n)\,dx_r,
	\end{align*}
which proves (ii).
\end{proof}

\begin{proposition}\label{prop_GL_is_Fubini}
Let $g$ be GL-Fubini on $\res{F}^n$, $a\in F^n$, $\g\in\Gamma^n$, $\delta\in\Gamma$; set $f=g^{a,\g}\,X^{\delta}$. Let $\tau\in\GL{n}{F}$; then $f\circ \tau$ is strongly Fubini on $F^n$, with
\[\int^{F^n} f\circ \tau(x)\,dx=\abs{\det{\tau}}^{-1}\int^{F^n} f(x) \,dx.\]
\end{proposition}
\begin{proof}
We claim that it suffices to prove the special case $a=0$, $\g=0$, $\delta=0$. This claim follows in the same way as the beginning of proposition \ref{prop_unipotent_is_Fubini}. Now assume $a=0$, $\g=0$, $\delta=0$.

Write $\tau=A U\Lambda$ as in lemma \ref{lemma_Iwasawa_decomposition}. Then $f\circ A=(g\circ\res{A})^0$ where $\res{A}$ is the image of $A$ in $\GL{n}{\res{F}}$; proposition \ref{proposition_integral_of_lifted_function} implies
\begin{align*}
    \int^{F^n}f\circ A(x)\,dx&=\int_{\res{F}^n}g\circ\res{A}(u)\,du\\
    &=\abs{\det{\res{A}}}^{-1}\int_{\res{F}^n} g(u)\,du\\
    &=\abs{\det{A}}^{-1} \int^{F^n} f(x)\,dx.
\end{align*}
Proposition \ref{prop_unipotent_is_Fubini} implies that $f^0\circ AU$ is strongly Fubini with
\[\int^{F^n}f\circ AU(x)\,dx=\int^{F^n}f\circ A(x)\,dx.\] Finally lemma \ref{lemma_translation_and_scaling_of_Fubini_functions} implies that $f\circ AU\Lambda$ is strongly Fubini, with \[\int^{F^n}f\circ AU\Lambda(x)\,dx=\abs{\det\Lambda}^{-1}\int^{F^n}f\circ A U(x)\,dx.\] Since $\det\tau=\det A\det\Lambda$, the proof is complete.
\end{proof}

The previous proposition extends by linearity to all of $\calL(F^n,\mbox{GL}_n)$ and so the main theorem is proved!

\begin{remark}
Suppose $F$ is a two-dimensional local field, with $\ROI{F}=\rho^{-1}(\roi{\res{F}})$ the rank two ring of integers. Assume that our chosen Haar measure on $\res{F}$ assigns $\roi{F}$ measure $1$. Then for any $\tau\in\GL{n}{F}$ and $a\in F^n$, the characteristic function of $a+\tau(\ROI{F}^n)$ belongs to $\calL(F^n)$, and \[\int^{F^n}\Char{a+\tau(\ROI{F}^n)}(x)\,dx=\abs{\det{\tau}}\in\Comp(X)=\CG.\]

Kim and Lee \cite{Kim-Lee1} have developed a measure theory on the algebra of sets generated by $\emptyset$, $F^n$ and $a+\tau(\ROI{F}^n)$ for $a\in F^n$, $\tau\in\GL{n}{F}$, under which $a+\tau(\ROI{F}^n)$ is given measure $\abs{\det{\tau}}$.

The main difference between the theory of Kim and Lee and that developed in this paper is that their measure takes values in a certain monoid of monomials rather than in a field; this allows Kim and Lee to extend their measure to the $\sigma$-algebra generated by the original algebra. Some parts of their theory may be recovered from this paper by 'taking leading terms' from $\Comp(X)$.
\end{remark}
%
%

\section{Invariant integral on $\GL{n}{F}$}
We will now consider integration on the space of matrices $\M{N}{F}$ and its unit group $\GL{N}{F}$.

Let $n=N^2$ and identify $\M{N}{F}$ with $F^n$ via an isomorphism $T:F^n\to\M{N}{F}$ of $F$ vector spaces. Let $\calL(\M{N}{F})$ be the $\CG$ space of $\CG$-valued functions $f$ on $\M{N}{F}$ for which $fT$ belongs to $\calL(F^n, \mbox{GL}_n)$; set \[\int^{\ssc{\M{N}{F}}}f(x)\,dx=\int^{F^n} fT(x)\,dx.\] 

\begin{remark}
The space $\calL(\M{N}{F})$ does not depend on the choice of the isomorphism $T$ since $\calL(F^n)$ is invariant under the action of $\GL{n}{F}$, and the functional $\int^{\ssc{\M{N}{F}}}$ depends on $T$ only up to a scaler multiple from $\abs{\mult{F}}=\{\lambda X^{\g}:\lambda\in\mult{\Comp},\g\in\Gamma\}$.

$\calL(\M{N}{N})$ is closed under translation, and $\int^{\ssc{\M{N}{F}}}$ is a translation invariant $\CG$-linear functional on the space.
\end{remark}

Of course, integrating on $\M{N}{F}$ is no harder than integrating on $F^n$. We are really interested in $\GL{N}{F}$:

\begin{definition}
Let $\calL(\GL{N}{F})$ denote the space of $\CG$-valued functions $\phi$ on $\GL{N}{F}$ such that $\tau\mapsto\phi(\tau)\abs{\det\tau}^{-n}$ extends to a function of $\calL(\M{N}{F})$.

The integral of $\phi$ over $\GL{N}{F}$ is defined by \[\int^{\ssc{\GL{N}{F}}} \phi(\tau)\,d\tau=\int^{\ssc{\M{N}{F}}} \phi(x)\abs{\det x}^{-n}\,dx,\] where the integrand on the right is really the extension of the function to $\M{N}{F}$.
\end{definition}

\begin{remark}
For the previous definition of the integral to be well defined, we must show that if $f_1,f_2\in\calL(\M{N}{F})$ are equal when restricted to $\GL{N}{F}$ then $f_1=f_2$.

It suffices to prove that if $f\in\calL(F^n)$ vanishes off some Zariski closed set (other than $F^n$), then $f$ is identically zero. By a \emph{locally constant} function $g$ on $F^n$, we mean a function such that for each $a\in F^n$, there exists $\g\in\Gamma$ such that, if $\varepsilon_1,\dots,\varepsilon_n\in F$ have valuation greater than $\g$, then $f(a_1+\epsilon_1,\dots,a_n+\epsilon_n)=f(a_1,\dots,a_n)$. If $g_1$, $g_2$ are locally constant, then so are $g_1+g_2$ and $g_1\circ A$ for any affine transformation of $F^n$. But a lifted function is locally constant and so any function in $\calL(F^n)$ is locally constant. It is now enough to show that if $p$ is a polynomial in $F[X_1,\dots,X_n]$, such that $p(\epsilon_1,\dots,\epsilon_n)=0$ whenever $\varepsilon_1,\dots,\varepsilon_n\in F$ have large enough valuation, then $p$ is the zero polynomial. This is easily proved by induction on $n$ and completes the proof.
\end{remark}

The integral is translation invariant:

\begin{proposition}\label{prop_trans_invariance_of_integral_on_GL}
Suppose $\phi$ belongs to $\calL(\GL{N}{F})$ and $\sigma\in\GL{N}{F}$. Then the functions $\tau\mapsto \phi(\sigma\tau)$ and $\tau\mapsto \phi(\tau\sigma)$ belong to $\calL(\GL{N}{F})$, with
\[\int^{\ssc{\GL{N}{F}}}\phi(\sigma\tau)\,d\tau = \int^{\ssc{\GL{N}{F}}} f(\tau)\,d\tau
    =\int^{\ssc{\GL{N}{F}}}\phi(\tau\sigma)\,d\tau\]
\end{proposition}
\begin{proof}
Let $r_{\sigma}$ (resp. $l_{\sigma}$) denote the element of $\GL{n}{F}$ (identified with $\mbox{GL}(\M{N}{F})$ via. $T$) defined by right (resp. left) multiplication by $\sigma$. Let $\tau\mapsto\phi(\tau)\abs{\det\tau}^{-n}$ be the restriction of $f\in\calL(\M{N}{F})$ to $\GL{N}{F}$, say. The function
\begin{align*}\tau\mapsto\, &\phi(\tau\sigma)\abs{\det{\tau}}^{-n}\\
    =&\abs{\det{\sigma}}^n\phi(\tau\sigma)\abs{\det{\tau\sigma}}^{-n}\\
    =&\abs{\det{\sigma}}^n
        \phi\circ r_{\sigma}(\tau)\abs{\det(r_{\sigma}\tau)}^{-n}
\end{align*}
is the restriction of $\abs{\det\sigma}^n f\circ r_{\sigma}\in\calL(\M{N}{F})$ to $\GL{N}{F}$.

Theorem \ref{theorem_main_result} therefore implies that
\begin{align*}
	\int^{\ssc{\GL{N}{F}}}\phi(\tau\sigma)d\tau
	&=\int^{\ssc{\M{N}{F}}}\abs{\det\sigma}^n f\circ r_{\sigma}(x)\,dx\\
	&=\abs{\det{\sigma}}^n\abs{\det r_{\sigma}}^{-1} \int^{\ssc{\M{N}{F}}}f(x)\,dx\\
	&=\abs{\det{\sigma}}^n\abs{\det r_{\sigma}}^{-1}  \int^{\ssc{\GL{N}{F}}} \phi(\tau)d\tau.
\end{align*}
Note that $\det{\sigma}$ is the determinant of $\sigma$ as an $N\times N$ matrix, and $\det r_{\sigma}$ is the determinant of $r_{\sigma}$ as an automorphism of the $N^2$-dimensional space $\M{N}{F}$.

To complete the proof for $r_{\sigma}$ it suffices to show that $\det{r_{\sigma}}=\det{\sigma}^n$. Let $e_{i,j}$ denote the $N\times N$ matrix with a $1$ in the $i,j$ position and zeros elsewhere. With respect to the ordered basis \[e_{1,1},e_{1,2},\dots,e_{1,N},e_{2,1},\dots,e_{2,N},\dots,e_{N,1},\dots,e_{N,N},\] $r(\sigma)$ acts as the block matrix
\[\left(\begin{array}{ccc}\sigma^t&&\\&\ddots&\\&&\sigma^t\end{array}\right),\]
($^t$ denotes transpose) which has determinant $\det{\sigma}^n$ as required.

The proof with $l_{\sigma}$ in place of $r_{\sigma}$ differs only in notation, except that one should use the ordered basis \[e_{1,1},e_{2,1},\dots,e_{N,1},e_{1,2},\dots,e_{N,2},\dots,e_{1,N},\dots,e_{N,N}.\]
\end{proof}

So we have obtained a translation invariant integral on the algebraic group $\GL{N}{F}$. Just as the integrals on $F$ and $F^n$ lift the usual Haar integral on $\res{F}$ and $\res{F}^n$, so too does this integral incorporate the Haar integral on $\GL{N}{\res{F}}$. To demonstrate this most clearly, it is prudent to now make the following assumptions on the chosen isomorphism $T$, which ensure a functoriality between our algebraic groups at the level of $\res{F}$ and at the level of $F$: 
\begin{enumerate}
\item $T$ restricts to an $\roi{F}$-linear isomorphism $\roi{F}^n\to\M{N}{\roi{F}}$.
\item There exists a $\res{F}$-linear isomorphism $\res{T}:\res{F}^n\to\M{N}{\res{F}}$ which makes the diagram commute:
	\[\begin{CD}
		\roi{F}^n	@>T>>		\M{N}{\roi{F}}	\\
		@V{}VV				@VV{}V		\\
		\res{F}^n	@>>{\res{T}}>	\M{N}{\res{F}}	
	\end{CD}\]
where the vertical arrows are coordinate-wise residue homomorphisms.
\end{enumerate}

\begin{remark}
These assumptions holds in particular if we identify $\M{N}{F}$ with $F^{n^2}$ in the most natural way, via the standard basis of $F^{n^2}$ and the basis of $\M{N}{F}$ used in proposition \ref{prop_trans_invariance_of_integral_on_GL}.
\end{remark}

Further, we now normalise the Haar measures on $\M{N}{\res{F}}$ and $\GL{N}{\res{F}}$ in the following way: give $\M{N}{\res{F}}$ the Haar measure obtained by pushing forward the product measure on $F^n$ via $T$, and then give $\GL{N}{\res{F}}$ the standard Haar measure $d_{\mbox{\scriptsize GL}_N}x=\det{x}^{-n}d_{\mbox{\scriptsize M}_N}x$. Such normalisations are not essential, but otherwise extraneous constants would appear in formulae below. It will be useful to call a complex-valued function on $\GL{N}{\res{F}}$ GL-Fubini if its pull back to $\res{F}^{n}$ via $\res{T}$ is GL-Fubini in the sense already defined. Again, note that a Schwartz-Bruhat function on $\M{N}{\res{F}}$ is certainly GL-Fubini.

We have already defined what is meant by the lift of a Haar integrable from $\res{F}$, $\res{F}^n$. Let us generalise this notion further:

\begin{definition}\label{definition_lift_to_algebraic_group}
Let $G$ denote either of the algebraic groups $\mbox{M}_N$, $\mbox{GL}_N$. Given a complex valued function $g$ on $G(\res{F})$, let $g^0$ be the complex valued function on $G(F)$ defined by
\begin{align*}
	g^0:F&\to\Comp\\
	x&\mapsto \begin{cases}
		g(\Res{x}) & x\in G(\roi{F}) \\
		0 & \mbox{otherwise.}
		\end{cases}
\end{align*}
\end{definition}

Then the compatibility between the integrals on $\mbox{M}_N$ at the level of $\res{F}$ and $F$ is the following:

\begin{proposition}
Suppose that $g$ is a complex valued Haar integrable function on $\M{N}{\res{F}}$ which is a GL-Fubini function on $\res{F}^n$ (eg. $g$ a Schwartz-Bruhat function on $\M{N}{\res{F}}$). Then $g^0$ belongs to $\calL(\M{N}{F})$, and \[\int^{\ssc{\M{N}{F}}} g^0(x)\,dx=\int_{\ssc{\M{N}{\res{F}}}} g(u)\,du.\]
\end{proposition}
\begin{proof}
By the existence of $\res{T}$ and its compatibility with $T$ we have and equality of functions on $\M{N}{F}$: \[ (g\res{T}^{-1})^0T=g^0.\] Definition of the integral on $\M{N}{F}$ implies \[\int^{\ssc{\M{N}{F}}}g^0(x)\,dx=\int^{F^n}(g\res{T}^{-1})^0(x)\,dx.\] Taking $G$ to be $n$ copies of the additive group, we showed in proposition \ref{proposition_integral_of_lifted_function} that the corresponding result to this one holds; so \[\int^{F^n}(g\res{T}^{-1})^0(x)\,dx=\int_{\res{F}^n}g\res{T}^{-1}(u)\,du.\] Finally, our normalisation of the Haar measure on $\M{N}{\res{F}}$ implies \[\int_{\res{F}^n}g\res{T}^{-1}(u)\,du=\int_{\M{N}{\res{F}}} g(u)\,du,\] which completes the proof.
\end{proof}

And now we prove the same result for $\mbox{GL}_N$:
\begin{proposition}
Suppose that $g$ is a complex valued Schwartz-Bruhat function on $\GL{N}{\res{F}}$ such that \[f(x)=\begin{cases} g(x)\abs{\det{x}}^{-n} & x\in\GL{N}{\res{F}} \\ 0 & \det{x}=0 \end{cases}\] is GL-Fubini on $\M{N}{\res{F}}$. Then $g^0$ belongs to $\calL(\GL{N}{F})$, and \[\int^{\ssc{\GL{N}{F}}} g^0(\tau)\,d\tau=\int_{\ssc{\GL{N}{\res{F}}}} g(u)\,du.\]
\end{proposition}
\begin{proof}
The assumption on $f$ and the previous proposition imply that $f^0$ belongs to $\calL(\M{N}{F})$. Moreover, $\tau\in\GL{N}{\roi{F}}$ implies \[f^0(\tau)=g(\res{\tau})\abs{\det{\res{\tau}}}^{-n}=g^0(\tau)\abs{\det{\tau}}^{-n},\] so that $f^0$ is an extension of $\tau\mapsto g^0(\tau)\abs{\det{\tau}}^{-n}$ from $\GL{N}{F}$ to a function in $\calL(\M{N}{F})$.

Therefore $g^0$ belongs to $\calL(\GL{N}{F})$ and
\begin{align*}
	\int^{\ssc{\GL{N}{F}}} g^0(\tau)\,d\tau
	&=\int^{\ssc{\M{N}{F}}} f^0(x)\,dx\\
	&=\int_{\ssc{\M{N}{\res{F}}}} f(u)\,du\\
	&=\int_{\ssc{\GL{N}{\res{F}}}} g(u)\,du
	\end{align*}
where the second equality follows from the previous proposition.
\end{proof}

\begin{remark}
Certainly if $g$ decreases sufficiently rapidly towards the boundary of $\GL{N}{\res{F}}$ in $\M{N}{\res{F}}$ then the hypothesis in the previous proposition will hold. In particular if $g$ is the restriction to $\GL{N}{\res{F}}$ of a Schwartz-Bruhat function on $\M{N}{\res{F}}$ and $s$ is complex with $\mbox{Re}(s)$ sufficiently large, then the result will hold for $g(\tau)\abs{\det{\tau}}^s$. This may be useful in generalising Jacquet-Godement theory \cite{Jacquet-Godement}.
\end{remark}

\section{Other algebraic groups and related problems}

\subsection{Integration over an arbitrary algebraic group}
Having established an integral on $\GL{N}{F}$, it would be useful to also be able to integrate on algebraic subgroups such as $\mbox{SL}_N(F)$ or $B_N(F)$, the group of invertible upper triangular matrices. Arguments similar to the above will surely provide such an integral, but to establish such results for an arbitrary algebraic group $G$ we require a more general abstract approach.

The author suspects that to each algebraic group $G$ their is a space of $\CG$-valued functions $\calL(G(F))$ on $G(F)$ and a linear functional $\int^{\ssc{G(F)}}$ on these functions with the following properties:

\begin{enumerate}
\item Compatibility between $\res{F}$ and $F$: if $g$ is a "nice" (eg. Schwartz-Bruhat) Haar integrable function on $G(\res{F})$, then $g^0$ (an obvious generalisation of definition \ref{definition_lift_to_algebraic_group}) belongs to $\calL(G(F))$ and \[\int^{\ssc{G(F)}} g^0(x)\,dx=\int_{\ssc{G(\res{F})}} g(u)\,du.\]
\item Translation invariance: if $f\in\calL(G(F))$ and $\tau\in G$, then $x\mapsto g(x\tau)$ is in $\calL(G(F))$, and \[\int^{\ssc{G(F)}}f(x\tau)\,dx=\int^{\ssc{G(F)}} f(x)\,dx.\]
\end{enumerate}

There should also be a left translation-invariant integral on $G(F)$, and this would coincide with the right-invariant integral if $G(\res{F})$ is unimodular.

Even for the simplest algebraic group $G=\mbox{"additive group"}$ these conditions are not enough to make the integral unique in a reasonable way; this is discussed in the first section of \cite{Morrow_1}. However, if we assume the existence of an absolute value which relates the integrals on $\mult{F}$ and $F$, the uniqueness does follow. We have observed a similar phenomenon in this paper where we constructed the integral on $F^n$ to be compatible with change of variables from $\GL{n}{F}$. So to ensure uniqueness we should add to the list the informal statement
\begin{enumerate}
\item[(iii)]  Compatibility between the integrals over different algebraic groups.
\end{enumerate}

\subsection{Subgroups of $\mbox{GL}_N$}

Once integration over algebraic subgroups of $\GL{N}{F}$ has been established, there are certain formulae which are expected to hold by analogy with the case of a local field. We quote two examples from \cite{Cartier}; for $f$ a complex-valued integrable function on $\GL{N}{\res{F}}$ (resp. on $\mbox{B}_N(\res{F})$),
\begin{align*}
\int_{\ssc{ \GL{N}{\res{F}} }} f(g)\,dg
	&=\int_{\ssc{ \GL{N}{\roi{\res{F}}} }}\int_{\ssc{ \mbox{B}_N(\res{F}) }} f(kb)\,dk\,d_R b \\
\int_{\ssc{\mbox{B}_N(\res{F})}} f(b)\,d_R b
	&=\int_{\Delta_N(\res{F})}\int_{\ssc{\mbox{U}_N(\res{F})}}f(u\lambda)\,du\,d\lambda,
\end{align*}
where $\mbox{U}_N$ denotes the u.u.t. matrices, $\Delta_N$ the diagonal matrices, and $d_R$ right Haar measure (apart from $\mbox{B}_N$, these groups are unimodular).

Writing these identities explicitly, one sees that these formulae require the class of integrable functions on $\GL{N}{F}$ to be invariant under certain polynomial changes of variables. It is therefore also important to extend the class of functions $\calL(F^n, \mbox{GL}_n)$ so that it is closed under certain polynomial changes of variables.

This is also precisely the sort of compatibility which may be important in (iii).

\subsection{Non-linear change of variables}
To develop integration on arbitrary algebraic groups and prove compatibility between them we are lead to investigate non-linear change of variables on $F^n$. A step in this direction is taken in \cite{Morrow_3} in the case of a two-dimensional local field (that is, $F$ is a complete discrete valuation field whose residue field is a local field). It is proved that if $f=g^{a,\g}$ is the lift to $F^2$ of a Schwartz-Bruhat function on $\res{F}^2$, and $h$ is a polynomial over $F$ which is not highly singular in a certain sense, then $(x,y)\mapsto f(x,y-h(x))$ is Fubini on $F^2$, and so \[\int^{F^2} f(x,y-h(x))\,dxdy=\int^{F^2} f(x,y-h(x))\,dydx=\int^{F^2} f(x,y)\,dydx.\] Note that the second equality follows simply from translation invariance of the integral.

However, the singularity assumption on $h$ is essential, for the author also proves the following:

\begin{proposition}
Suppose $F$ is a two-dimensional local field and $\res{F}$ has finite characteristic $p$. Let $h(X)=t^{-1}X^p$ and let $f$ be \emph{any} Schwartz-Bruhat function on $K\times K$. Then for all $y\in F$, the function $x\mapsto f^0(x,y-h(x))$ is integrable, with $\int^F f^0(x,y-h(x))\,dx=0$. Therefore \[\int^F\int^F f^0(x,y-h(x))\,dxdy=0,\] whereas \[\int^F\int^F f^0(x,y-h(x))\,dydx=\int\int f(u,v)\,dvdu,\] which need not be zero.
\end{proposition}

Whether this will cause a problem in verifying existence of integral on algebraic groups is unclear to the author. If such "wild" changes of variable do not appear when changing charts on one's algebraic group, then this may not be too serious. However, it is certainly an unexpected result; it appears to be a measure-theoretic consequence of the characteristic $p$ local field $\res{F}$ being imperfect. See \cite{Morrow_3} for further discussion.

\mbox{}\\
	Matthew T. Morrow\\
	Maths and Physics Building,\\
	University of Nottingham,\\
	University Park,\\
	Nottingham\\
	NG7 2RD\\
	United Kingdom\\
   	\email{matthew.morrow@maths.nottingham.ac.uk}


\begin{thebibliography}{99}

\bibitem{Bump}
{\sc D.~Bump}, {\em Automorphic forms and representations}, vol.~55 of
  Cambridge Studies in Advanced Mathematics, Cambridge University Press,
  Cambridge, 1997.

\bibitem{Cartier}
{\sc P.~Cartier}, {'Representations of {$p$}-adic groups: a survey'}, in
  Automorphic forms, representations and $L$-functions (Proc. Sympos. Pure
  Math., Oregon State Univ., Corvallis, Ore., 1977), Part 1, Proc. Sympos. Pure
  Math., XXXIII, Amer. Math. Soc., Providence, R.I., 1979, pp.~111--155.

\bibitem{Fesenko-analysis-on-arithmetic-schemes}
{\sc I.~Fesenko}, {'Analysis on arithmetic schemes. {I}'}, Doc. Math.,
  (2003), pp.~261--284 (electronic).
\newblock Kazuya Kato's fiftieth birthday.

\bibitem{Fesenko-vector-spaces}
{\sc I.~Fesenko}, {'Measure and integration on vector spaces over two
  dimensioal local fields'},  (2005).
\newblock available at {\tt www.maths.nott.ac.uk/personal/ibf/mp.html}.

\bibitem{Fesenko-analysis-on-loop-spaces}
{\sc I.~Fesenko}, {'Measure, integration and elements of harmonic analysis
  on generalized loop spaces'}, in Proceedings of the St. Petersburg
  Mathematical Society. Vol. XII, vol.~219 of Amer. Math. Soc. Transl. Ser. 2,
  Providence, RI, 2006, Amer. Math. Soc., pp.~149--165.

\bibitem{Hrushovski-Kazhdan1}
{\bibname E.~Hrushovski \and D.~Kazhdan}, {'Integration in valued fields'}, in
  Algebraic geometry and number theory, vol.~253 of Progr. Math., Birkh\"auser
  Boston, Boston, MA, 2006, pp.~261--405.

\bibitem{Hrushovski-Kazhdan2}
{\bibname E.~Hrushovski \and D.~Kazhdan}, {'The value ring of geometric motivic
  integration and the {I}wahori {H}ecke algebra of $\mbox{SL}_2$'}, {\tt
  arXiv:math.AG/0609115},  (2006).

\bibitem{IHLF}
{\bibname I.~Fesenko \and M.~Kurihara}, Eds., {\em Invitation to higher local
  fields}, vol.~3 of Geometry \& Topology Monographs, Geometry \& Topology
  Publications, Coventry, 2000.
\newblock Papers from the conference held in M\"unster, August 29--September 5,
  1999.

\bibitem{Jacquet-Godement}
{\bibname R.~Godement \and H.~Jacquet}, {\em Zeta Functions of Simple Algebras}, vol.~260 of Lecture Notes in Mathematics, Springer-Verlag, Berlin, 1972.

\bibitem{Kim-Lee2}
{\bibname H.~H. Kim \and K.-H. Lee}, {'Spherical {H}ecke algebras of {$\rm SL\sb
  2$} over 2-dimensional local fields'}, Amer. J. Math., 126 (2004),
  pp.~1381--1399.

\bibitem{Kim-Lee1}
{\bibname H.~H. Kim \and K.-H. Lee}, {'An invariant measure on $\mbox{GL}_n$ over
  $2$-dimensional local fields'}, University of Nottingham Mathematics preprint
  series,  (2005).

\bibitem{Morrow_1}
{\sc M.~T. Morrow}, {'Integration on valuation fields over local fields'}. arXiv:math.NT/07122172

\bibitem{Morrow_3}
{\sc M.~T. Morrow}, {'Fubini's theorem and non-linear change of variables over a two-dimensional local field'.} arXiv:math.NT/07122177
\end{thebibliography}
\end{document}